\documentclass[12pt]{amsart}
\usepackage{amsmath, amssymb, amsfonts, amsthm, mathtools}
\usepackage{hyperref}
\usepackage{tabularx}
\usepackage{booktabs}
\usepackage{caption}
\usepackage{aurical}
\usepackage{color}
\usepackage{comment}

\newtheorem{thm}{Theorem}[section]
\newtheorem{lem}[thm]{Lemma}
\newtheorem{cor}[thm]{Corollary}
\newtheorem{prop}[thm]{Proposition}
\newtheorem{ex}[thm]{Example}

\newtheorem*{prob*}{Open problem}

\theoremstyle{definition}

\newtheorem{defi}[thm]{Definition}

\theoremstyle{remark}

\newtheorem{rem}[thm]{Remark}
\newtheorem*{rem*}{Remark}

\DeclareMathOperator{\id}{id}

\DeclareMathOperator{\Hom}{Hom}

\newcommand{\kringel}{\mathbin{\raise0.5pt\hbox{$\scriptstyle\circ$}}}
\newcommand{\pkt}{\mathbin{\raise0.5pt\hbox{$\scriptstyle\bullet$}}}
\newcommand{\sq}{\mathbin{\raise0.5pt\hbox{$\scriptscriptstyle\square$}}}

\newcommand{\C}{\mathbb{C}}

\newcommand{\tr}{\mathop{\rm tr}}
\newcommand{\ad}{{\rm ad}}

\newcommand{\Der}{{\rm Der}}

\newcommand{\Lg}{\mathfrak{g}}
\newcommand{\Lh}{\mathfrak{h}}    \newcommand{\Lk}{\mathfrak{k}}

\newcommand{\Ll}{\mathfrak{l}}

\newcommand{\Lr}{\mathfrak{r}}
\newcommand{\Ls}{\mathfrak{s}}

\newcommand{\Lz}{\mathfrak{z}}

\renewcommand{\phi}{\varphi}

\begin{document}

\title{Deformations of semi-direct products}

\author[M. A. Alvarez]{Mar\'ia Alejandra Alvarez}
\author[S. Rivi\`ere]{Salim Rivi\`ere}
\author[N. Rojas]{Nadina Rojas}
\author[S. Vera]{Sonia Vera}
\author[F. Wagemann]{Friedrich Wagemann}

\date{\today}


\begin{abstract}
We exhibit in this article a contraction of the direct product Lie algebra $\Lg\oplus\Lg$ of a finite-dimensional complex Lie algebra $\Lg$ onto the semi-direct product Lie algebra $\Lg\rtimes\Lg$, where the first factor $\Lg$ is viewed as a trivial Lie algebra and as the adjoint $\Lg$-module. This contraction gives rise to a non-zero cohomology class in the second cohomology space. We generalize to the setting of $\Lh\oplus\Lg$ and $\Lh\rtimes\Lg$ with respect to a given crossed module of Lie algebras $\Lh\to\Lg$. We give many examples to illustrate our results.
\end{abstract}

\maketitle

\section*{Introduction}

In this article, we investigate deformations and cohomology of semi-direct product Lie algebras, i.e. of Lie algebras of
the form $V\rtimes\Lg$ where $\Lg$ is a Lie algebra and $V$ an $\Lg$-module. All Lie algebras are supposed to be finite-dimensional and complex. We will study at first in more detail the special case where
$V=\Lg$ is the Lie algebra, viewed as an $\Lg$-module via the adjoint action (and as an abelian Lie algebra).
Later on, we will also consider the more general setting $\Lh\rtimes\Lg$, when there is given a crossed module $\mu:\Lh\to\Lg$. In the corresponding semi-direct product, the Lie algebra $\Lh$ will be viewed as a $\Lg$-module and abelian Lie algebra though.

Motivation for our work comes from the cohomology computations of Richardson \cite{Ri} and Rauch \cite{Ra}, work which was then further developed by Burde-Wagemann \cite{BW}.
Namely, first Richardson and then Rauch in a more systematic way computed the cohomology of the semi-direct products of the form $V\rtimes\Ls\Ll_2(\C)$ of $\Ls\Ll_2(\C)$  with an irreducible $\Ls\Ll_2(\C)$-module $V$.
They found in particular a $1$-dimensional $H^2(V\rtimes\Ls\Ll_2(\C),V\rtimes\Ls\Ll_2(\C))$ for certain odd-dimensional $V$, including the adjoint module $V=\Ls\Ll_2(\C)$ of dimension $3$. Starting from these results, we show more generally that the semi-direct product $\Lg\rtimes\Lg$ of a finite-dimensional complex Lie algebra $\Lg$ with itself, viewed as a module with respect to the adjoint representation (and as an abelian Lie algebra),
deforms to the direct product Lie algebra $\Lg\oplus\Lg$, see  Proposition \ref{prop_contraction}.
This semi-direct product $\Lg\rtimes\Lg$ arises naturally as the Lie algebra of the Lie group $TG=\Lg\rtimes G$, where $G$ is the connected, simply connected Lie group integrating $\Lg$.

Let us give an overview over the content of this article. The deformation of Proposition \ref{prop_contraction} is given by a non-trivial element in $H^2(\Lg\rtimes\Lg,\Lg\rtimes\Lg)$, see Proposition \ref{prop_non_trivial_cocycle}. In Section 3, we study the possible isomorphy of the direct product $\Lg\oplus\Lg$ and the semi-direct product $\Lg\rtimes\Lg$. We use the derivation Lie algebras of $\Lg\oplus\Lg$ and $\Lg\rtimes\Lg$ in order to conclude that these are isomorphic only for the abelian Lie algebra $\Lg$, see Corollary \ref{cor_non_isomorphy}. Along the way, we compute for example the first cohomology space $H^1$ with adjoint coefficients of the semi-direct product, see Theorem \ref{thm_cohomology_semidirect_prod}.

In Section 3, we ask the question how our constructions from Sections 1-2 generalize to an arbitrary semi-direct product $\Lh\rtimes\Lg$ for a Lie algebra $\Lg$ acting by derivations on some Lie algebra $\Lh$. It turns out that in case there exists a crossed module $\mu:\Lh\to\Lg$, many results from Sections 1-2 generalize: The cochain defining the deformation is still a $2$-cocycle (Proposition \ref{prop_cocycle_crmod}) and the cocycle is still non-trivial (Theorem \ref{thm_non_trivial_cocycle_crmod}) in case the crossed module $\mu:\Lh\to\Lg$ satisfies $\mu([\Lh,\Lh])\not=0$. We investigate functoriality of the $2$-cocycle and show some universality property of the cocycle on $\Lg\rtimes\Lg$, see Theorem \ref{thm_universality}.

Section 4 is devoted to the comparison of the three Lie algebras
at stake: the direct product Lie algebra $\Lh\oplus\Lg$, the semi-direct product $\Lh\rtimes\Lg$ (where $\Lh$ is considered as an abelian Lie algebra) and the deformed Lie algebra $\Lh\times^\mu_t\Lg$. We give examples
where these three Lie algebras are isomorphic or not. For this, we determine the center and the derived ideal of these Lie algebras. One general criterion is here that in case the center of $\Der(\Lg)$ is non-trivial, the deformed algebra cannot be isomorphic to the direct product, see Corollary \ref{cor_criterion_center}. We show in Theorem \ref{Teo 2 step} that for a $2$-step nilpotent Lie algebra $\Lg$ and the crossed module $\ad:\Lg\to\Der(\Lg)$, the deformed Lie algebra is not isomorphic to the direct product. Using the derived ideal, we show that the deformed Lie algebra $\Lh\times^\mu_t\Lg$ is solvable if and only if $\Lg$ is solvable, see Corollary \ref{cor_criterion_derived_ideal}. In Section 4.4, we try to obtain the derivation algebras of the semi-direct product and the deformed Lie algebra, but we do not get an answer which is as nice as in the case of $\Lg\rtimes\Lg$.

\vspace{.5cm}

\noindent{\bf Acknowledgements:} This collaboration is funded by the MATH-AmSud Program, Project VARLIE codes 23-MATH-16 and AMSUD230039. We are thankful for this support which financed a stay of FW and MAA in Córdoba in August 2024 where a large part of this work was done.


\section{Cocycle and deformation}

Let $\Lg$ be a complex finite-dimensional non-abelian Lie algebra. We consider the semidirect product $\Lg\rtimes\Lg$ of $\Lg$ with itself viewed as a module with respect to the adjoint representation (and as an abelian Lie
algebra), i.e. with product defined by
\[[(m,g),(m',g')]:=([g,m']-[g',m],[g,g']).\]

Observe that the bracket of $\Lg$ can serve as a $2$-cochain for the semi-direct product $\Lg\rtimes\Lg$ with values in the adjoint representation in the following way:
$$c((m,g),(m',g'))=(0,[m,m']).$$


Then we obtain:

\begin{lem}
$c$ is a $2$-cocycle on the semi-direct product $\Lg\rtimes\Lg$ with adjoint values.
\end{lem}

\begin{proof}
On the one hand
$$c([(m,g),(m',g')],(m'',g''))=(0,[[g,m'],m'']-[[g',m],m'']).$$
On the other hand,
$$[(m,g),c((m',g'),(m'',g''))]=[(m,g),(0,[m',m''])]=([[m',m''],m],[g,[m',m'']]).$$
By writing out the cyclic permutations of these terms, one sees that $c$ satisfies the cocycle identity in the first component due to the Jacobi identity for $[[m',m''],m]$, and in the second component due to the Jacobi identity
for terms of the form $[g,[m',m'']]$.
\end{proof}

The next proposition is rather surprising.

\begin{prop}  \label{prop_non_trivial_cocycle}
The cocycle $c$ cannot be cohomologous to zero.
\end{prop}

\begin{proof}
Indeed, a coboundary is of the form
$$d\alpha((m,g),(m',g'))=[(m,g),\alpha(m',g')]-[(m',g'),\alpha(m,g)]-\alpha([(m,g),(m',g')]).$$
Denoting $\alpha:\Lg\rtimes\Lg\to\Lg\rtimes\Lg$ by $\alpha(m,g)=(\alpha_1(m,g),\alpha_2(m,g))$, we get
\begin{eqnarray*}
d\alpha((m,g),(m',g'))&=&(g\cdot\alpha_1(m',g')-\alpha_2(m',g')\cdot m,[g,\alpha_2(m',g')])\\
&-&(g'\cdot\alpha_1(m,g)-\alpha_2(m,g)\cdot m',[g',\alpha_2(m,g)])-\alpha(g\cdot m'-g'\cdot m,[g,g']).
\end{eqnarray*}
Now put $g=g'=0$. Then we obtain
$$d\alpha((m,0),(m',0))=(-\alpha_2(m',0)\cdot m - \alpha_2(m,0)\cdot m',0).$$
Therefore we see that we can never obtain $c((m,0),(m',0))=(0,[m,m'])$ as a coboundary.
\end{proof}

This gives us the following result.

\begin{cor}
For any complex finite-dimensional non-abelian Lie algebra $\Lg$, we have
$$H^2(\Lg\rtimes\Lg,\Lg\rtimes\Lg)\not= 0.$$
\end{cor}

\subsection{Deformation and contractions}\label{ss:def_cont}

The linear deformation related to the cocycle $c$ is given by:
\begin{equation}   \label{contraction}
[(m,g),(m',g')]_t:=(g\cdot m'-g'\cdot m,[g,g']+t[m,m']).
\end{equation}

Consider now the linear map 
$\psi_t:\Lg\oplus\Lg\to\Lg\oplus\Lg$ defined by
$$\psi_t(m,g):=(g-\sqrt{t}\,\,m,g+\sqrt{t}\,\,m),$$
for all $m,g\in\Lg$
.
In matrix form, $\psi_t$ reads
$$\psi_t=\left(\begin{array}{cc} -\sqrt{t} & 1 \\ \sqrt{t} & 1 \end{array}\right). $$
It is thus invertible for $t\in\C\setminus\{0\}$ and its inverse is
$$\psi_t^{-1}=\frac{1}{-2\sqrt{t}}\left(\begin{array}{cc} 1 & -1 \\ -\sqrt{t} & -\sqrt{t} \end{array}\right).$$

Let $[\cdot,\cdot]_\oplus$ be the direct product Lie bracket. A direct computation gives then
$$\psi_t^{-1}([\psi_t(m,g),\psi_t(m',g')]_{\oplus})=(g\cdot m'-g'\cdot m,[g,g']+t[m,m'])=[(m,g),(m',g')]_t,$$
where it is clear from equation \eqref{contraction} that the limit for $t\to 0$ gives the semi-direct product bracket. On the other hand, for all $t\not=0$, the Lie algebra $(\Lg\oplus\Lg,[\cdot,\cdot]_t)$, given by the linear deformation associated with the cocycle $c$, is isomorphic to the direct product Lie algebra $(\Lg\oplus\Lg,[\cdot,\cdot]_\oplus)$ via the isomorphism $\psi_t$. We therefore get a contraction of the direct product Lie algebra $\Lg\oplus\Lg$
onto the semi-direct product Lie algebra $\Lg\rtimes\Lg$. 


\begin{prop}  \label{prop_contraction}
For any (finite-dimensional complex) Lie algebra $\Lg$, the direct product Lie algebra $\Lg\oplus\Lg$ contracts onto the semi-direct product Lie algebra $\Lg\rtimes\Lg$.
\end{prop}

\begin{rem}
It is well known that contractions are not unique. For instance, we can obtain a different contraction from the direct product Lie algebra $\Lg\oplus\Lg$ onto the semi-direct product Lie algebra $\Lg\rtimes\Lg$ by using the following linear map $\phi_t:\Lg\oplus\Lg\to\Lg\oplus\Lg$ for $t\in\C\setminus\{0\}$:
$$\phi_t(m,g)=(g+tm,g)$$
for $g\in\Lg$ (Lie algebra side) and $m\in\Lg$ (module side). The difference between this contraction and the previous one, is that the latter arises from a deformation associated with a coboundary, namely,  $\tilde{c}((m,g),(m',g'))=([m,m'], 0)$.
\end{rem}

\section{Isomorphy of direct and semi-direct products}

We ask ourselves in this section whether the Lie algebras $\Lg\oplus\Lg$ and  $\Lg\rtimes\Lg$ are non-isomorphic in general. Observe that for both Lie algebras, we have
$$H^0(\Lg\oplus\Lg,\Lg\oplus\Lg)\cong H^0(\Lg\rtimes\Lg,\Lg\rtimes\Lg)\cong Z(\Lg)\oplus Z(\Lg),$$
so the $0$th cohomology space with adjoint coefficients cannot distinguish them.

In order to obtain an answer, we consider the derivation algebras of the Lie algebras $\Lg\oplus\Lg$
and $\Lg\rtimes\Lg$ and compute the first cohomology spaces. Recall that we consider in the semi-direct product the module term as the adjoint $\Lg$-module, but as a trivial Lie algebra. This differs from the setting in \cite{RAS}.

\subsection{Lie algebras of derivations}\label{ss:lie_der}

\begin{prop}  \label{prop_direct_prod}
Every derivation $D:\Lg\oplus\Lg\to\Lg\oplus\Lg$ (of the direct product Lie algebra $\Lg\oplus\Lg$) is of the
matrix form
$$D=\left(\begin{array}{cc} D_1 & D_2 \\ D_3 & D_4 \end{array}\right),$$
where $D_1$ and $D_4$ are derivations of $\Lg$ and where $D_2,D_3\in\Hom(\Lg/[\Lg,\Lg],Z(\Lg))$ are otherwise arbitrary.
\end{prop}

\begin{proof}
We compute both sides of $D[(m,g),(m',g')]=[D(m,g),(m',g')]+[(m,g),D(m',g')]$ for all $m,m'\in\Lg$ (module part) and all $g,g'\in\Lg$ (Lie algebra part).
$$D[(m,g),(m',g')]=D([m,m'],[g,g'])=(D_1[m,m']+D_2[g,g'],D_3[m,m']+D_4[g,g']).$$
On the other hand
\begin{eqnarray*}
[D(m,g),(m',g')]&=&[(D_1(m)+D_2(g),D_3(m)+D_4(g)),(m',g')] \\
&=&([D_1(m)+D_2(g),m'],[D_3(m)+D_4(g),g']).
\end{eqnarray*}
and the last term reads
\begin{eqnarray*}
[(m,g),D(m',g')]&=&[(m,g),(D_1(m')+D_2(g'),D_3(m')+D_4(g'))] \\
&=&([m,D_1(m')+D_2(g')],[g,D_3(m')+D_4(g')]).
\end{eqnarray*}
Thus we obtain the following two relations for $D$ to be a derivation:
\begin{itemize}
\item $D_1[m,m']+D_2[g,g']=[D_1(m)+D_2(g),m']+[m,D_1(m')+D_2(g')]$, and
\item $D_3[m,m']+D_4[g,g']=[D_3(m)+D_4(g),g']+[g,D_3(m')+D_4(g')]$.
\end{itemize}
From these two relations, we obtain the following six conditions on $D_1$, $D_2$, $D_3$ and $D_4$ by putting first $g=g'=0$, then $g=m'=0$ and finally $m=m'=0$:
\begin{itemize}
\item[(1)] $D_1[m,m']=[D_1(m),m']+[m,D_1(m')]$, i.e. $D_1\in\Der(\Lg)$,
\item[(2)] $D_3[m,m']=0$, i.e. $D_3$ vanishes on $[\Lg,\Lg]$,
\item[(3)] $0=[m,D_2(g')]$, i.e. $D_2$ has values in the center $Z(\Lg)$ of $\Lg$,
\item[(4)] $0=[D_3(m),g']$, i.e. $D_3$ has values in $Z(\Lg)$,
\item[(5)] $D_2[g,g']=0$, i.e. $D_2$ vanishes on $[\Lg,\Lg]$ and
\item[(6)] $D_4[g,g']=[D_4(g),g']+[g,D_4(g')]$, i.e. $D_4\in\Der(\Lg)$.
\end{itemize}
\end{proof}

\begin{prop}  \label{prop_semi_direct_prod}
Every derivation $D:\Lg\rtimes\Lg\to\Lg\rtimes\Lg$ (of the semi-direct product Lie algebra $\Lg\rtimes\Lg$) is of the
matrix form
$$D=\left(\begin{array}{cc} D_1 & D_2 \\ D_3 & D_4 \end{array}\right),$$
where $D_2$ and $D_4$ are derivations of $\Lg$, where $D_3\in\Hom(\Lg/[\Lg,\Lg],Z(\Lg))$  and where the bracket with $D_4$ describes the failure of $D_1$ to be in the centroid, i.e. for all $x,y\in\Lg$, we have
$$D_1([x,y])=[D_4(x),y]+[x,D_1(y)].$$
\end{prop}

\begin{proof}
We compute both sides of $D[(m,g),(m',g')]=[D(m,g),(m',g')]+[(m,g),D(m',g')]$ for all $m,m'\in\Lg$ (module part) and all $g,g'\in\Lg$ (Lie algebra part).
\begin{eqnarray*}
D[(m,g),(m',g')]&=&D([g,m']-[g',m],[g,g'])\\
&=&(D_1([g,m]-[g',m])+D_2[g,g'],D_3([g,m']-[g',m])+D_4[g,g']).
\end{eqnarray*}
On the other hand
\begin{eqnarray*}
[D(m,g),(m',g')]&=&[(D_1(m)+D_2(g),D_3(m)+D_4(g)),(m',g')] \\
&=&([D_3(m)+D_4(g),m']-[g',D_1(m)+D_2(g)],[D_3(m)+D_4(g),g']).
\end{eqnarray*}
and the last term reads
\begin{eqnarray*}
[(m,g),D(m',g')]&=&[(m,g),(D_1(m')+D_2(g'),D_3(m')+D_4(g'))] \\
&=&([g,D_1(m')+D_2(g')]-[D_3(m')+D_4(g'),m],[g,D_3(m')+D_4(g')]).
\end{eqnarray*}
Again we obtain two relations for $D$ to be a derivation:
\begin{itemize}
\item $D_1([g,m']-[g',m])+D_2[g,g']=[D_3(m)+D_4(g),m']-[g',D_1(m)+D_2(g)]+[g,D_1(m')+D_2(g')]-[D_3(m')+D_4(g'),m]$, and
\item $D_3([g,m']-[g',m])+D_4[g,g']=[D_3(m)+D_4(g),g']+[g,D_3(m')+D_4(g')]$.
\end{itemize}
From these two relations, we obtain the following five conditions on $D_1$, $D_2$, $D_3$ and $D_4$ by putting first $g=g'=0$, then $g=m'=0$ and finally $m=m'=0$:
\begin{itemize}
\item[(1)] $0=[D_3(m),m']-[D_3(m'),m]$,
\item[(2)] $D_1[g',m]=[g',D_1(m)]+[D_4(g'),m]$
\item[(3)] $D_3[g',m]=[g',D_3(m)]$, i.e. $D_3$ is in the centroid,
\item[(4)] $D_2[g,g']=-[g',D_2(g)]+[g,D_2(g')]$, i.e. $D_2\in\Der(\Lg)$, and
\item[(5)] $D_4[g,g']=[D_4(g),g']+[g,D_4(g')]$, i.e. $D_4\in\Der(\Lg)$.
\end{itemize}
Notice that Equations (1) and (3) imply
$$D_3[x,y]=[x,D_3(y)]=[y,D_3(x)]=-D_3[y,x]=-[y,D_3(x)].$$
Therefore we conclude that $D_3$ vanishes on $[\Lg,\Lg]$. But this implies in turn with (3) that the image of $D_3$ is central.
\end{proof}

\begin{lem}  \label{lemma_semi_direct_prod}
The vector space
$$X:=\{(D_1,D_4)\,|\,D_4\in\Der(\Lg)\,\,\,{\rm and}\,\,\,\forall x,y\in\Lg:\,D_1([x,y])=[D_4(x),y]+[x,D_1(y)]\},$$
which appears in the previous proposition, can be described more explicitly as
$$X\cong\Der(\Lg)\oplus{\rm Centr}(\Lg),$$
where
$${\rm Centr}(\Lg):=\{\varphi:\Lg\to\Lg\,|\,\forall x,y \in\Lg:\,\varphi([x,y])=[x,\varphi(y)]\}$$
is the centroid of $\Lg$ (see \cite{LL}).
\end{lem}

\begin{proof}
Observe that $X$ contains the two subspaces
$$\{(D_4,D_4)\,|\,D_4\in\Der(\Lg)\}\cong\Der(\Lg)$$
and
$$\{(D_1,0)\,|\,D_1\in{\rm Centr}(\Lg)\}\cong{\rm Centr}(\Lg).$$
The intersection of these two subspaces is clearly reduced to $\{(0,0)\}$. Observe furthermore that for $(D_1,D_4)\in X$, we have that $D_1-D_4\in{\rm Centr}(\Lg)$. Indeed, for all $x,y\in\Lg$, we have
\begin{eqnarray*}
(D_1-D_4)([x,y])&=&D_1([x,y])-D_4([x,y]) \\
&=&[D_4(x),y]+[x,D_1(y)]-[D_4(x),y]-[x,D_4(y)] \\
&=&[x,(D_1-D_4)(y)],
\end{eqnarray*}
as $D_4$ is a derivation. Denote this element of the centroid by $D_1-D_4=:C$. Then we obtain in $X$ the decomposition
$$(D_1,D_4)=(C+D_4,D_4)=(C,0)+(D_4,D_4).$$
This means that every element of $X$ can be written as the sum of elements from the two subspaces and concludes the proof that $X$ is the direct sum of these two subspaces.
\end{proof}

\subsection{Consequences}

From Proposition \ref{prop_direct_prod}, one deduces

\begin{prop}
The cohomology space of the direct product is given by
$$H^1(\Lg\oplus\Lg,\Lg\oplus\Lg)=(H^1(\Lg,\Lg))^{\oplus 2}\oplus(\Hom(\Lg/[\Lg,\Lg],Z(\Lg)))^{\oplus 2}.$$
\end{prop}

On the other hand, from Proposition \ref{prop_semi_direct_prod} and the explicit description of $X$ in Lemma \ref{lemma_semi_direct_prod}, we obtain

\begin{thm}  \label{thm_cohomology_semidirect_prod}
The cohomology space of the semi-direct product is given by
$$H^1(\Lg\rtimes\Lg,\Lg\rtimes\Lg)=(H^1(\Lg,\Lg))^{\oplus 2}\oplus\Hom(\Lg/[\Lg,\Lg],Z(\Lg))\oplus{\rm Centr}(\Lg).$$
\end{thm}

\begin{proof}
Indeed, the inner derivation given by the element $(m,g)$ is in matrix notation
$$\ad_{(m,g)}=\left(\begin{array}{cc} \ad_g & \ad_m \\ 0 & \ad_g \end{array}\right).$$
Therefore, only the factor $\{(D_4,D_4)\,|\,D_4\in\Der(\Lg)\}$ gets divided by $\ad_\Lg$, the centroid factor remains unchanged.
\end{proof}

\begin{cor} \label{cor_non_isomorphy}
Let $\Lg$ be a non-abelian Lie algebra.
Then the direct product $\Lg\oplus\Lg$ and the semi-direct product $\Lg\rtimes\Lg$ are not isomorphic.
\end{cor}

\begin{proof}
Observe that $\Hom(\Lg/[\Lg,\Lg],Z(\Lg))\subset{\rm Centr}(\Lg)$. But $\id_\Lg\in{\rm Centr}(\Lg)$, and $\id_\Lg$ is not in $\Hom(\Lg/[\Lg,\Lg],Z(\Lg))$ as soon as $\Lg$ is non-abelian.
\end{proof}

As an example, note that for a simple finite-dimensional complex Lie algebra $\Lg$, the centroid ${\rm Centr}(\Lg)$ is the field of complex numbers $\C$ by the Schur Lemma and $\Hom(\Lg/[\Lg,\Lg],Z(\Lg))=0$, thus we obtain in this case
$$\Der(\Lg\rtimes\Lg)=\Der(\Lg\oplus\Lg)\oplus\{\lambda \id_\Lg\,|\,\lambda\in\C\}.$$


\section{Deformations of semi-direct products related to crossed modules}

\subsection{The cocycle}
In this section, we go from the semi-direct product $\Lg\rtimes\Lg$ to semi-direct products of the form $\Lh\rtimes\Lg$. It turns out that the construction of the cocycle $c$ works in case there is a crossed module
$\mu:\Lh\to\Lg$.

\begin{defi}
A Lie algebra homomorphism $\mu:\Lh\to\Lg$ together with an action of $\Lg$ on $\Lh$ by derivations is a {\it crossed module of Lie algebras} in case we have for all $g\in\Lg$ and all $h,h'\in\Lh$
\begin{itemize}
\item[(a)] $\mu(g\cdot h)=[g,\mu(h)]$ and
\item[(b)] $\mu(h)\cdot h'=[h,h']$.
\end{itemize}
\end{defi}

\begin{prop} \label{prop_cocycle_crmod}
Let $\mu:\Lh\to\Lg$ be a crossed module of Lie algebras. Then the $2$-cochain $c\in C^2(\Lh\rtimes\Lg,\Lh\rtimes\Lg)$ defined by
$$c((h,g),(h',g')):=(0,\mu([h,h']))$$
is a $2$-cocycle.
\end{prop}

\begin{proof}
We compute the cocycle identity for all $h,h',h''\in\Lh$ and all $g,g',g''\in\Lg$. We obtain using Property (a) of a crossed module from the second to the third line
\begin{eqnarray*}
c([(h,g),(h',g')],(h'',g''))&=&c((g\cdot h'-g'\cdot h,[g,g']),(h'',g''))\\
&=&(0,\mu([g\cdot h'-g'\cdot h,h'']))\\
&=&(0,[[g,\mu(h')],\mu(h'')]-[[g',\mu(h)],\mu(h'')]).
\end{eqnarray*}
On the other hand, we obtain using Property (b) of a crossed module from the second to the third line
\begin{eqnarray*}
[(h,g),c((h',g'),(h'',g''))]&=&[(h,g),(0,\mu([h',h'']))]\\
&=&(-\mu([h',h''])\cdot h,[g,\mu([h',h''])])\\
&=&(-[[h',h''],h],[g,[\mu(h'),\mu(h'')]]).
\end{eqnarray*}
Using cyclic permutations of $(h,h',h'')$ and $(g,g',g'')$, we thus obtain zero in the first component by the Jacobi identity in $\Lh$, and zero in the second component by the Jacobi identity in $\Lg$ on elements of the form
$[g,[\mu(h'),\mu(h'')]]$.
\end{proof}

Observe that in the special case where the crossed module is $\id_\Lg:\Lg\to\Lg$, the construction reduces to the construction of Section 1.

\begin{thm}  \label{thm_non_trivial_cocycle_crmod}
Let $\Lg$, $\Lh$ be Lie algebras such that there exists a crossed module $\mu:\Lh\to\Lg$ with $\mu([\Lh,\Lh])\not=0$. Then
$$H^2(\Lh\rtimes\Lg,\Lh\times\Lg)\not=0.$$
\end{thm}

Observe that the hypothesis and the statement are related by the fact that the semi-direct product is built out of the action which is given by the crossed module $\mu:\Lh\to\Lg$. The hypothesis excludes for example abelian Lie algebras $\Lh$. Observe furthermore that in the semi-direct product, the vector space $\Lh$ is considered as a $\Lg$-module, but as an abelian Lie algebra.

\begin{proof}
The $2$-cochain considered in the previous proposition generates in fact a non-trivial cohomology class. This follows exactly like in Proposition \ref{prop_non_trivial_cocycle}: The cocycle can possibly be obtained as a coboundary only if the second component is zero. But this is the case only if $\mu([\Lh,\Lh])=0$ which is excluded.
\end{proof}

\subsection{The contraction}

As before, we obtain a contraction associated to this cocycle/defor\-ma\-tion. Define $\phi_t:\Lh\oplus\Lg\to\Lh\oplus\Lg$ by $\phi_t(h,g):=(t^{-1/2}h,g)$ for all $h\in\Lh$ and all $g\in\Lg$. The linear map $\phi_t$ is clearly invertible for $t\not=0$. Then for the bracket
$$[(h,g),(h',g')]_1:=(g\cdot h'-g'\cdot h,[g,g']+\mu([h,h'])),$$
we compute
$$\phi_t[\phi_t^{-1}(h,g),\phi_t^{-1}(h',g')]_1.$$
We obtain
\begin{eqnarray*}
\phi_t[\phi_t^{-1}(h,g),\phi_t^{-1}(h',g')]_1&=&\phi_t[(t^{1/2}h,g),(t^{1/2}h',g')]_1 \\
&=&\phi_t(t^{1/2}(g\cdot h'-g'\cdot h),[g,g']+\mu([h,h'])) \\
&=&(g\cdot h'-g'\cdot h,[g,g']+t\mu([h,h']))
\end{eqnarray*}
Thus the bracket $[\cdot,\cdot]_t$ contracts onto the semi-direct product bracket in the limit $t\to 0$. This also shows that all the Lie algebras which we obtain for different values of $t\not=0$ are isomorphic. We will provide some criteria to identify the deformed Lie algebra in the next section.

\subsection{Functoriality and Universality}

Let us briefly study the functoriality of the class $[c]$ associated to a crossed module $\mu:\Lh\to\Lg$
with respect to morphisms of crossed modules $(\phi,\psi):(\mu:\Lh\to\Lg)\to(\mu':\Lh'\to\Lg')$.

\begin{defi}
A {\it morphism of crossed modules} $(\phi,\psi):(\mu:\Lh\to\Lg)\to(\mu':\Lh'\to\Lg')$ is a pair of Lie algebra morphisms $\phi:\Lh\to\Lh'$ and $\psi:\Lg\to\Lg'$ such that $\mu'\circ\phi=\psi\circ\mu$ and for all $g\in\Lg$ and all $h\in\Lh$,
$$\phi(g\cdot h)=\psi(g)\cdot\phi(h).$$
\end{defi}

The deformed bracket behaves covariantly with respect to a morphism:

\begin{lem}
Let $\mu:\Lh\to\Lg$ be a crossed module of Lie algebras and $(\phi,\psi):(\mu:\Lh\to\Lg)\to(\mu':\Lh'\to\Lg')$ a morphism of crossed modules. Then the deformed bracket
$$[(h,g),(h',g')]_1=(g\cdot h'-g'\cdot h,[g,g']+\mu([h,h']))$$
on $\Lh\oplus\Lg$ is sent to the deformed bracket on $\Lh'\oplus\Lg'$ via $(\phi,\psi)$.
\end{lem}

\begin{proof}
Let $h,h'\in\Lh$ and $g,g'\in\Lg$ and compute:
\begin{eqnarray*}
(\phi,\psi)[(h,g),(h',g')]_1&=&(\phi,\psi)(g\cdot h'-g'\cdot h,[g,g']+\mu([h,h'])) \\
&=&(\phi(g\cdot h'-g'\cdot h),\psi[g,g']+\psi\circ\mu([h,h'])) \\
&=&(\psi(g)\cdot\phi(h')-\psi(g')\cdot\phi(h),[\psi(g),\psi(g')]+\mu'[\phi(h),\phi(h')]) \\
&=&[(\phi(h),\psi(g)),(\phi(h'),\psi(g'))]_1.
\end{eqnarray*}
\end{proof}

On the other hand, the cohomology classes behave contravariantly with respect to a morphism: Recall that for a cocycle $c\in C^k(\Lg',M')$, the induced cocycle for a morphism $\psi:\Lg\to\Lg'$ is the composition
$$\Lambda^k\Lg\stackrel{\Lambda^k \psi}{\longrightarrow}\Lambda^k\Lg'\stackrel{c}{\to} M'.$$
This induces a map $\psi^*:C^k(\Lg',M')\to C^k(\Lg,(M')^\psi)$, where $(M')^\psi$ is the $\Lg$-module with underlying vector space $M'$ and $\Lg$-action via $\psi$. The map $\psi^*$ induces a map in cohomology still denoted $\psi^*$.

\begin{lem}
Let $(\phi,\psi):(\mu:\Lh\to\Lg)\to(\mu':\Lh'\to\Lg')$ be a morphism of crossed modules. Then the
$2$-cocycle $c'$, associated to $\mu':\Lh'\to\Lg'$ and defined for all $h,h'\in\Lh'$ and all $g,g'\in\Lg'$ by
$$c'((h,g),(h',g'))=\mu'([h,h'])$$
on the semi-direct product $\Lh'\times^{\mu'}_1\Lg'$, is sent to the $2$-cocycle $\psi^*(c)$, image of the cocycle associated to $\mu:\Lh\to\Lg$, via $(\phi,\psi)$.
\end{lem}

\begin{proof}
Let $h,h'\in\Lh$ and $g,g'\in\Lg$. We compute
\begin{eqnarray*}
(\phi,\psi)(c')((h,g),(h',g'))&=&\mu'([\phi(h),\phi(h')]) \\
&=&\mu'\circ\phi([h,h']) \\
&=&\psi\circ\mu([h,h']) \\
&=& \psi(c)([h,h']).
\end{eqnarray*}
\end{proof}

It is clear that this cocycle has values in the factor $\Lg$ (resp. $\Lg'$) of the semi-direct product $\Lh\rtimes\Lg$ (resp. $\Lh'\rtimes\Lg'$). One can therefore view it also as $c'\in C^2(\Lh'\rtimes\Lg',\Lg')$ being sent to
$(\phi,\psi)(c')\in C^2(\Lg\rtimes\Lg,\Lg')$, where $\Lg'$ is viewed as a $\Lh\rtimes\Lg$-module via $\psi$. This has the advantage that the module does not change.

\begin{thm}  \label{thm_universality}
The cocycle $c\in C^2(\Lg\rtimes\Lg,\Lg\rtimes\Lg)$, defined by
$$c((h,g),(h',g'))=[h,h'],$$
is universal in the sense that it induces the cocycle
$$c((h,g),(h',g'))=\mu([h,h'])$$
for any crossed module $\mu:\Lh\to\Lg$ via a morphism $(\phi,\psi):(\mu:\Lh\to\Lg)\to(\id:\Lg\to\Lg)$.
\end{thm}

\begin{proof}
The morphism $(\phi,\psi):(\mu:\Lh\to\Lg)\to(\id:\Lg\to\Lg)$ is given by $(\phi,\psi)=(\mu,\id_\Lg)$. It sends the
cocycle
$$c((h,g),(h',g'))=[h,h']$$
to the cocycle
$$(\phi,\psi)(c)((h,g),(h',g'))=[\phi(h),\phi(h')]=\phi([h,h'])=\mu([h,h']).$$
\end{proof}

Here again, one may restrict the coefficient module to $\Lg$ in order to have the same coefficient module in the domain and in the range.


\section{Identifying the deformed Lie algebra}

The problem we address in this section is to identify the deformed Lie algebra $(\Lh\oplus\Lg,[\cdot,\cdot]_t)$ associated to the crossed module $\mu:\Lh\to\Lg$. We have seen before that for the crossed module $\id_\Lg:\Lg\to\Lg$, this bracket is isomorphic to the direct product bracket. Our main concern in this section is to decide whether the deformed Lie algebra is or is not isomorphic to the direct product Lie algebra and to the semi-direct product Lie algebra.

\subsection{The case where \texorpdfstring{$\mu$}{}
is an isomorphism}

\begin{prop}\label{Prop.iso}
In case the map $\mu$ in the crossed module $\mu:\Lh\to\Lg$ is an isomorphism, the deformed Lie algebra
$(\Lh\oplus\Lg,[\cdot,\cdot]_t)$ is isomorphic to the direct product Lie algebra $\Lh\oplus\Lg$.
\end{prop}

\begin{proof}
Observe that we have a morphism of crossed modules from $\mu:\Lh\to\Lg$ to $\id_\Lh:\Lh\to\Lh$ induced by
$(\id_\Lh,\mu^{-1})$. Indeed, it makes the square commutative: $\mu^{-1}\circ\mu=\id_\Lh\circ\id_\Lh$. Moreover,
it commutes with the actions. For this, we have to show for all $g\in\Lg$ and all $h\in\Lh$ that
\begin{equation}   \label{*}
\mu^{-1}(g)\cdot h=g\cdot h
\end{equation}
(where the action on the LHS is in the crossed module $\id_\Lh:\Lh\to\Lh$ and on the RHS in the crossed module $\mu:\Lh\to\Lg$). But this is clear by applying the isomorphism $\mu$ to both sides. This means that we have an isomorphism of crossed modules.

Thus the constructions of the deformed Lie algebra must correspond under this isomorphism. Indeed, for the crossed module $\mu:\Lh\to\Lg$, we have the bracket
$$[(h,g),(h',g')]_t':=(g\cdot h'-g'\cdot h,[g,g']+t\mu([h,h'])),$$
and for the crossed module $\id_\Lh:\Lh\to\Lh$, we have the bracket
$$[(h,k),(h',k')]_t'':=(k\cdot h'-k'\cdot h,[k,k']+t[h,h']).$$
We claim that the linear map $\varphi:=\id_\Lh+\mu:(\Lh\oplus\Lh,[\cdot,\cdot]_t'')\to(\Lh\oplus\Lg,[\cdot,\cdot]_t')$ is an isomorphism of Lie algebras. For this, we compute on the one hand
\begin{eqnarray*}
[\varphi(h,k),\varphi(h',k')]_t'&=&[(h,\mu(k)),(h',\mu(k'))] \\
&=&(\mu(k)\cdot h'-\mu(k')\cdot h,\mu[k,k']+t\mu([h,h'])).
\end{eqnarray*}
And on the other hand
\begin{eqnarray*}
\varphi([(h,k),(h',k')]_t'')&=&\varphi(k\cdot h'-k'\cdot h,[k,k']+t[h,h']) \\
&=&(k\cdot h'-k'\cdot h,\mu[k,k']+t\mu([h,h'])).
\end{eqnarray*}
In fact, both expressions are equal, because the actions in the two crossed modules correspond via Equation \eqref{*}. As the deformed Lie algebra $(\Lh\oplus\Lh,[\cdot,\cdot]_t'')$ is isomorphic to the direct product for $\id_\Lh:\Lh\to\Lh$, it must be the same for the deformed bracket $[\cdot,\cdot]_t'$ built from the crossed module $\mu:\Lg\to\Lh$.
\end{proof}

We have now an example which shows that $[\cdot,\cdot]_t$ can be isomorphic to the direct product even if $\mu$ is not an isomorphism.

\begin{ex}\label{ex injective}
Consider a simple Lie algebra $\Lg$ and the crossed module given by the inclusion $\Lg\to\Lg\oplus\Lg$ of the ideal $\Lg$ into the first factor of $\Lg\oplus\Lg$. The deformed bracket $[\cdot,\cdot]_t$ on the semi-direct product is then
$$[(h,(g_1,g_2)),(h',(g_1',g_2'))]_t=([g_1,h']-[g',h],([g_1,g_1']+t[h,h'],[g_2,g_2'])).$$
It is clear that on the first two factors of the resulting $\Lg\oplus\Lg\oplus\Lg$, we have the identity crossed module and can therefore use the isomorphism with the direct product in order to show that the deformed Lie algebra
$(\Lg\oplus\Lg\oplus\Lg,[\cdot,\cdot]_t)$ is isomorphic to the direct product Lie algebra $\Lg\oplus\Lg\oplus\Lg$.
\end{ex}


More generally, we have the following result on crossed module between semi-simple Lie algebras.


\begin{prop}
A crossed module $\mu:\Lh\to\Lg$ between semi-simple Lie algebras $\Lh$ and $\Lg$ is the sum of crossed modules of the kind $0:V\to\Lk$, $\{0\}\hookrightarrow\Lk$  and of crossed modules where $\mu$ is an isomorphism.
\end{prop}

\begin{proof}
Indeed, the kernel of $\mu$ must be an ideal, thus up to restriction onto a certain number of simple factors of $\Lh$, we can suppose that $\mu$ is injective. The restriction process does not alter the situation, we have just excluded a certain number of crossed modules of the form $0:V\to\Lg$ where $V$ is a $\Lg$-module.
But the image of $\mu$ must be also an ideal, thus, up to excluding a certain number of factors of the semi-simple Lie algebra $\Lg$, $\mu$ is surjective.
\end{proof}

In fact, a crossed module $\mu:\Lh\to\Lg$ where $\Lh$ is semi-simple is, up to restriction onto some simple factors, injective and thus the sum of crossed modules of the kind $0:V\to\Lk$ and of crossed modules where $\mu$ is an isomorphism. In the same vein, a crossed module $\mu:\Lh\to\Lg$ where $\Lg$ is semi-simple is, up to restriction to a certain number of simple factors in $\Lg$, surjective, thus a central extension. But this central extension splits by semi-simplicity, thus again, it is the sum of crossed modules of the kind $0:V\to\Lk$ and of crossed modules where $\mu$ is an isomorphism.

Observe that the crossed modules of the type $0:V\to\Lg$ for a $\Lg$-module $V$ give rise to the usual semi-direct product and they are thus not isomorphic to the direct product, if the action is non-trivial.

\begin{ex}
Let $\Lr_2=\langle x,y\,|\,[x,y]=x\rangle$ be the $2$-dimensional solvable (non nilpotent) Lie algebra. Consider the deformed algebra $\Lr_2\times^\mu_t(\Lr_2\oplus\Lr_2)$, where $\mu:\Lr_2\to\Lr_2\oplus\Lr_2$ is the crossed module given by the inclusion map into the first factor of the direct product. In order to see the isomorphism with the direct product, we choose the basis $X_1=(x,(0,0))$, $X_2=(y,(0,0))$, $X_3=(0,(x,0))$, $X_4=(0,(y,0))$, $X_5=(0,(0,x))$, $X_6=(0,(0,y))$. The commutation relations for the deformed Lie algebra (with $t=1$) read
$$[X_1,X_2]=X_3,\,\,[X_2,X_3]=-X_1,\,\,[X_1,X_4]=X_1,\,\,[X_3,X_4]=X_3,\,\,[X_5,X_6]=X_5.$$
The commutation relations for the direct product read simply $[X_1,X_2]=X_1$, $[X_3,X_4]=X_3$, $[X_5,X_6]=X_5$. When explicitly writing out the isomorphism $\psi=\psi_1$ from Section \ref{ss:def_cont}, one obtains
$$\psi(X_1)=X_3-X_1,\,\,\psi(X_2)=X_4-X_2,\,\,\psi(X_3)=X_1+X_3,\,\,\psi(X_4)=X_2+X_4,$$
while $\psi(X_5)=X_5$ and $\psi(X_6)=X_6$. One computes that (like in Section \ref{ss:def_cont}) this linear map is an isomorphism from the deformed Lie algebra to the direct product Lie algebra.
\end{ex}

\begin{prop}
In case the map $\mu$ in the crossed module $\mu: \Lh \to \Lg$ is surjective then 
$(\Lh\oplus(\Lh/\ker(\mu)),[\cdot,\cdot]_t')$ is isomorphic to the deformed Lie algebra
$(\Lh\oplus\Lg,[\cdot,\cdot]_t)$, where 
$$[(h,\overline{k}),(h',\overline{k'})]_t':= \left(k\cdot h'-k'\cdot h,\left[\overline{k},\overline{k'}\right]+ t\left[\overline{h},\overline{h'}\right]\right)$$
for all $h, h', k, k' \in \Lh$.
\end{prop}

\begin{proof}
We first observe that $[\cdot,\cdot]_t'$ is well defined since $\ker(\mu) \subseteq \Lz(\Lh)$. 
On the other hand, since $\mu$ is an epimorphism there exists an isomorphism $\overline{\mu}: \Lh/\ker(\mu) \to \Lg$ such that 
$\overline{\mu}(\overline{h})= \mu(h)$ for all $h \in \Lh$.

The linear map $\varphi:=\id_\Lh+\overline{\mu}:(\Lh\oplus(\Lh/\ker(\mu)),[\cdot,\cdot]_t') \to  (\Lh\oplus\Lg,[\cdot,\cdot]_t)$ is an isomorphism of Lie algebras and the proof is similar to the one given in Proposition \ref{Prop.iso}, which completes the proof.
\end{proof}

\subsection{Identifying the deformed Lie algebra using the center}

Let us compute the center of the deformed Lie algebra $\Lh\times^\mu_t\Lg:=(\Lh\oplus\Lg,[\cdot,\cdot]_t)$ where
$$[(h,g),(h',g')]_t=(g\cdot h'-g'\cdot h,[g,g']+t\mu([h,h'])).$$

\begin{prop}\label{centerdeformed}
The center $Z(\Lh\times^\mu_t\Lg)$ of $\Lh\times^\mu_t\Lg$ is given by
$$\{(h,g)\,|\,g'\cdot h=0\ \forall g'\in\Lg,\ [h,h']\in\ker(\mu)\ \forall h'\in\Lh,\ g\cdot h'=0\ \forall h'\in\Lh,\ g\in Z(\Lg)\}.$$
\end{prop}

\begin{proof}
Consider for a fixed pair $(h,g)\in\Lh\oplus\Lg$ the condition
$$[(h,g),(h',g')]_t=(g\cdot h'-g'\cdot h,[g,g']+t\mu([h,h']))=(0,0)$$
for all $(h',g')\in\Lh\oplus\Lg$. Taking first $(h',g')=(h',0)$, we get $g\cdot h'=0\ \forall h'\in\Lh$ in the first component and $[h,h']\in\ker(\mu)$ in the second component. Taking then $(h',g')=(0,g')$, we get $g'\cdot h=0\ \forall g'\in\Lg$ in the first component and $g\in Z(\Lg)$ in the second component.
\end{proof}

The next example shows that the direct product Lie algebra and the deformed Lie algebra could be non-isomorphic even for $\mu$ injective.

\begin{ex}
Let $\mathfrak{g}= \langle e_1, e_2, e_3, \dots, e_{n}\rangle$ defined by $[e_1, e_i]= e_i$ for $i=4, \dots, n$ and $[e_2, e_3]= e_3$. Consider the subalgebra $\mathfrak{h}= \langle e_1, e_3, \dots, e_{n}\rangle$ of $\mathfrak{g}$ and the crossed module given by the inclusion $i: \Lh \to \Lg$ of the ideal $\Lh$ into  $\Lg$.

It is easy to see that the center of the direct product Lie algebra $\Lh \oplus \Lg$ is $\langle e_3 \rangle \oplus 0$ but the center of the deformed Lie algebra $(\Lh\oplus\Lg,[\cdot,\cdot]_t)$ is $0$ for all $t$.
\end{ex}

The following example shows that $[\cdot,\cdot]_t$ may not be isomorphic to the direct product Lie algebra if $\mu$ is surjective.

\begin{ex}
Let $\Lh= \langle x_1, x_2, x_3, x_{12}, x_{13}, x_{23}\,|\,[x_i, x_j]= x_{ij} \text{ with } 1 \leq i < j \leq 3 \rangle$ be the free $2$-step nilpotent Lie algebra of rank $3$ and let $I= \langle x_{12}, x_{13} \rangle$ be an ideal of $\Lh$.
Consider the deformed algebra $\Lh \times^\mu_t (\Lh/I)$, where $\mu: \Lh \to \Lh/I$ is the crossed module given by the quotient map.

It is clear that the
$$
\Lz((\Lh \oplus (\Lh/I), [\cdot,\cdot]_{\oplus}))= \langle (x_{12}, 0), (x_{13}, 0), (x_{23}, 0), (0, \mu(x_1)), (0, \mu(x_{23}))\rangle
$$
and from Proposition \ref{centerdeformed} we obtain
$$
\Lz((\Lh \oplus (\Lh/I), [\cdot,\cdot]_{t}))= \langle (x_{12}, 0), (x_{13}, 0), (x_{23}, 0), (0, \mu(x_{23}))\rangle.
$$
It follows that the deformed Lie algebra $\Lh \times^\mu_t (\Lh/I)$ is not isomorphic to the direct product Lie algebra
$\Lh \oplus (\Lh/I)$ for all $t$.
\end{ex}

Consider now the crossed modules of the kind $\ad:\Lg\to\Der(\Lg)$. For this kind of crossed module, the previous proposition specializes to

\begin{cor}  \label{cor_center}
Let $\Lg\times^\ad_t\Der(\Lg)$ be the deformed Lie algebra corresponding a crossed module of the form $\ad:\Lg\to\Der(\Lg)$. Then the center is given by
$$Z(\Lg\times^\ad_t\Der(\Lg))=\{(h,0)\,|\,h\in\bigcap_{D\in\Der(\Lg)}\ker(D),\,\,\,\ad_{[h,h']}=0\ \forall h'\in\Lg\}.$$
In particular, we have
$$Z(\Lg\times^\ad_t\Der(\Lg))\subset Z(\Lg)\oplus\{0\}.$$
\end{cor}

\begin{proof}
For the first statement, it suffices to see that for $(h,D)\in\Lg\oplus\Der(\Lg)$, we have $D\cdot h=D(h)$ by definition of the action in the crossed module $\ad:\Lg\to\Der(\Lg)$. Thus $D\cdot h=0$ for all $h\in\Lg$ means that $D=0$.

For the second statement, it is enough to see that
$$h\in\bigcap_{D\in\Der(\Lg)}\ker(D)\subset\bigcap_{\ad_g\in\Der(\Lg)}\ker(\ad_g)=Z(\Lg).$$
\end{proof}

\begin{cor}  \label{cor_criterion_center}
Let $\Lg\times^\ad_t\Der(\Lg)$ be the deformed Lie algebra corresponding a crossed module of the form $\ad:\Lg\to\Der(\Lg)$. Suppose that $Z(\Der(\Lg))\not=0$. Then $\Lg\times^\ad_t\Der(\Lg)$ is not isomorphic to the direct product Lie algebra $\Lg\oplus\Der(\Lg)$.
\end{cor}

\begin{proof}
Indeed, the center of the direct product Lie algebra $\Lg\oplus\Der(\Lg)$ is
$$Z(\Lg\oplus\Der(\Lg))=Z(\Lg)\oplus Z(\Der(\Lg)),$$
which is not isomorphic to $Z(\Lg\times^\ad_t\Der(\Lg))$ under the stated hypothesis.
\end{proof}

\medskip
 
We note the following example where we show that the center $Z(\Lg)$ may be strictly bigger than the space $\bigcap_{D\in\Der(\Lg)}\ker(D)$.

\begin{ex}
Let $\Lg$ be a $2$-step nilpotent Lie algebra with derived ideal $\Lg'\subset Z(\Lg)$, but $\Lg'\not= Z(\Lg)$. Thus $\Lg$ has an abelian factor. Let $\Lg=\langle e_1,\ldots,e_n\rangle\oplus\langle v_1,\ldots,v_m\rangle$ where $\langle e_1,\ldots,e_n\rangle=Z(\Lg)$, $\langle v_1,\ldots,v_m\rangle=:\nu$ with $[\nu,\nu]=\Lg'$ and where
$e_1\in Z(\Lg)$ is not in $\Lg'$. Let $D$ be the endomorphism of $\Lg$ which is $1$ on $e_1$ and zero on all other elements. Then $D$ is a derivation (because $e_1\not\in\Lg'$ and $e_1\in Z(\Lg)$) such that $e_1\not\in\ker(D)$, but $e_1\in Z(\Lg)$. Therefore we have a strict inclusion
$$\bigcap_{D\in\Der(\Lg)}\ker(D)\subset\bigcap_{g\in\Lg}\ker(\ad_g)=Z(\Lg).$$
\end{ex}

\begin{lem}\label{center2step}
Let $\mathfrak{g}:= [\mathfrak{g}, \mathfrak{g}] \oplus \mathcal{V}$ be a $2$-step nilpotent Lie algebra, then
$$D_0= \left(\begin{array}{cc} 2 I & 0 \\ 0 & I \end{array}\right),$$
is a derivation of $\mathfrak{g}$.
\end{lem}

\begin{proof}
Let $x, y \in \mathfrak{g}$, then there exist $x_1, y_1 \in [\mathfrak{g}, \mathfrak{g}]$ and $x_2, y_2 \in \mathcal{V}$ such that
$$
x= x_1 + x_2 \;\;\; \text{ and } y= y_1 + y_2.
$$
Since $[x, y]= [x_2, y_2]$, we obtain $D_0[x, y]= 2 [x_2, y_2]$. On the other hand
\begin{itemize}
\item $[D_0(x), y]= [2 x_1 + x_2, y_1 + y_2]= [x_2, y_2]$
\item $[x, D_0(y)]= [x_1 + x_2, 2y_1 + y_2]= [x_2, y_2]$
\end{itemize}
This completes the proof.
\end{proof}

\begin{lem}\label{centertrivial}
Let $\mathfrak{g}$ be a $2$-step nilpotent Lie algebra and
let $\Lg\times^\ad_t\Der(\Lg)$ be the deformed Lie algebra corresponding a crossed module of the form $\ad:\Lg\to\Der(\Lg)$. Then
$Z(\Lg\times^\ad_t\Der(\Lg))= 0$.
\end{lem}

\begin{proof}
This follows from Corollary \ref{cor_center} and Lemma \ref{center2step}.
\end{proof}

\begin{thm}\label{Teo 2 step}
Let $\mathfrak{g}$ be a $2$-step nilpotent Lie algebra and let $\Lg\times^\ad_t\Der(\Lg)$ be the deformed Lie algebra corresponding a crossed module of the form $\ad:\Lg\to\Der(\Lg)$. Then $\Lg\times^\ad_t\Der(\Lg)$ is not isomorphic to the direct product Lie algebra $\Lg\oplus\Der(\Lg)$.
\end{thm}

\begin{proof}
This follows from Corollary \ref{cor_center} and Lemma \ref{centertrivial}.
\end{proof}

\begin{rem}
From Theorem \ref{Teo 2 step}, we obtain that $\Lg\times^\ad_t\Der(\Lg)$ is not a nilpotent Lie algebra.
\end{rem}

\begin{ex}
Let $\Lh_{2n+1}$ be the Heisenberg Lie algebra. Consider the crossed module $\ad:\Lh_{2n+1}\to\Der(\Lh_{2n+1})$. The associated deformed Lie algebra is just the semi-direct product Lie algebra, because the cocycle $c((h,g),(h',g'))=(0,\ad_{[h,h']})$ for all $h,h'\in\Lh_{2n+1}$ is trivial (as $[h,h']$ is central). Let us show that semi-direct product (i.e. the deformed algebra) $\Lh_{2n+1}\rtimes\Der(\Lh_{2n+1})$ and direct product $\Lh_{2n+1}\oplus\Der(\Lh_{2n+1})$ are not isomorphic.

The derivations $\Der(\Lh_{2n+1})$ of $\Lh_{2n+1}$ are all of the matrix form (see for instance \cite{GK})
$$\left(\begin{array}{cc} A & 0 \\ * & \tr(A) \end{array}\right),$$
where $A$ is of size $2n\times 2n$. We therefore obtain derivations $D_i=E_{ii}+E_{2n+1, 2n+1}$. The kernel of $D_i$ consists of elements $\,^t(x_1,\ldots,x_{i-1},0,x_{i+1},\ldots,x_{2n},0)$ and is of dimension $2n-1$. The intersection of the kernels of all derivations $D_i$ for $i=1,\ldots,2n$ is zero. Therefore we obtain using Corollary \ref{cor_center}
$$Z(\Lh_{2n+1}\rtimes\Der(\Lh_{2n+1}))=\{0\}.$$
On the other hand, the center of $\Lh_{2n+1}$ is $1$-dimensional, and thus direct product and semi-direct product cannot be isomorphic.
\end{ex}



\subsection{Identifying the deformed Lie algebra using the derived ideal}

Let us compute here the derived Lie algebra $[\Lh\times^\mu_t\Lg,\Lh\times^\mu_t\Lg]$ of the Lie algebra $\Lh\times^\mu_t\Lg$ associated to the crossed module $\mu:\Lh\to\Lg$.

\begin{prop}
The derived Lie algebra $[\Lh\times^\mu_t\Lg,\Lh\times^\mu_t\Lg]$ of the Lie algebra $\Lh\times^\mu_t\Lg$ associated to the crossed module $\mu:\Lh\to\Lg$ is
$$[\Lh\times^\mu_t\Lg,\Lh\times^\mu_t\Lg]={\mathcal O}\oplus[\Lg,\Lg],$$
where ${\mathcal O}$ is the subspace of $\Lh$ generated by the action of $\Lg$ on all elements of $\Lh$.
\end{prop}

\begin{proof}
Indeed, all elements of the form $(0,x)$ for $x\in[\Lg,\Lg]$ are in $[\Lh\times^\mu_t\Lg,\Lh\times^\mu_t\Lg]$, because
$$[(0,g),(0,g')]=(0,[g,g']).$$
On the other hand, all elements of the form $(y,0)$ for $y\in{\mathcal O}$, the subspace of $\Lh$ generated by the action of $\Lg$ on all elements of $\Lh$, are in $[\Lh\times^\mu_t\Lg,\Lh\times^\mu_t\Lg]$, because
$$[(0,g),(h',0)]=(g\cdot h',0).$$
\end{proof}

We draw the readers attention to the fact that ${\mathcal O}\subset\Lh$ is not an ideal or even a subalgebra in general. But ${\mathcal O}\oplus[\Lg,\Lg]\subset \Lh\times^\mu_t\Lg$ is an ideal. One can deduce from the previous proposition that $\Lh\times^\mu_t\Lg$ is not isomorphic to the direct product $\Lh\oplus\Lg$ in case
${\mathcal O}$ is not isomorphic to $[\Lh,\Lh]$. But we will rather consider the derived series:

\begin{cor}  \label{cor_criterion_derived_ideal}
The Lie algebra $\Lh\times^\mu_t\Lg$ is solvable if and only if $\Lg$ is solvable.
\end{cor}

\begin{proof}
Indeed, observe that iterating the previous proposition, we have that the second component of  $[[\Lh\times^\mu_t\Lg,\Lh\times^\mu_t\Lg],[\Lh\times^\mu_t\Lg,\Lh\times^\mu_t\Lg]]$ is $[[\Lg,\Lg],[\Lg,\Lg]]$
and so on. Therefore if $\Lh\times^\mu_t\Lg$ is solvable, the second component will be zero after a finite number of steps, i.e. $\Lg$ is solvable.
Conversely, in the series
$$(\Lh\times^\mu_t\Lg,[\Lh\times^\mu_t\Lg,\Lh\times^\mu_t\Lg],[[\Lh\times^\mu_t\Lg,\Lh\times^\mu_t\Lg],[\Lh\times^\mu_t\Lg,\Lh\times^\mu_t\Lg]],\ldots),$$
it is first $\Lg$, then $[\Lg,\Lg]$, then $[[\Lg,\Lg],[\Lg,\Lg]]$ etc which acts on the first component to obtain the orbit of the next step. Therefore, if $\Lg$ is solvable, eventually it is the zero space acting on the first component, i.e. the first component gives also zero in the next step. Therefore $\Lh\times^\mu_t\Lg$ is solvable.
\end{proof}

Observe that for the analogous statement with {\it nilpotent}, one needs (for the nilpotency of the deformed Lie algebra on top of $\Lg$ being nilpotent) that the action of $\Lg$ on $\Lh$ is nilpotent, because here one operates at each step with the whole Lie algebra $\Lg$.

\begin{ex}
Consider the crossed module $\ad:\Lg\to\Der(\Lg)$ where $\Lg$ is the $3$-dimensional solvable (not nilpotent) Lie algebra $\Lg=\Lr_{3,1}=\langle e_1,e_2,e_3\,|\,[e_1,e_2]=e_2,\,\,[e_1,e_3]=e_3\rangle$. The Lie algebra of derivations of $\Lg$ is $6$-dimensional and given by all matrices (written in the basis $(e_1,e_2,e_3)$) of the form
$$\left(\begin{array}{ccc} 0 & 0 & 0 \\ * & * & * \\ * & * & * \end{array}\right).$$
So $\Der(\Lg)$ is generated by the elementary matrices $(E_{21},E_{22},E_{23},E_{31},E_{32},E_{33})$. The deformed Lie algebra $\Lg\times^\ad_t\Der(\Lg)$ has the generators $(X_1=(e_1,0), X_2=(e_2,0),X_3=(e_3,0),X_4=(0,E_{21}),X_5=(0,E_{22}),X_6=(0,E_{23}),X_7=(0,E_{31}),X_8=(0,E_{32}),X_9=(0,E_{33}))$. The commutation relations between these generators are:
\begin{itemize}
\item[] $[X_1,X_2]=-X_4$, $[X_1,X_3]=-X_7$,
\item[] $[X_2,X_5]=-X_2$, $[X_2,X_8]=-X_3$,
\item[] $[X_3,X_6]=-X_2$, $[X_3,X_9]=-X_3$
\item[] $[X_4,X_5]=-X_4$, $[X_4,X_8]=-X_7$,
\item[] $[X_5,X_6]=X_6$, $[X_5,X_8]=-X_8$,
\item[] $[X_6,X_7]=X_4$, $[X_6,X_8]=X_5-X_9$, $[X_6,X_9]=X_6$,
\item[] $[X_7,X_9]=-X_7$,
\item[] $[X_8,X_9]=-X_8$.
\end{itemize}
One can read off from these commutation relations the successive derived ideals of the deformed Lie algebra $\Lg\times^\ad_t\Der(\Lg)$. We obtain
$$(\Lg\times^\ad_t\Der(\Lg))'=\langle X_2,X_3,X_4,X_5-X_9, X_6,X_7,X_8\rangle$$
and
$$(\Lg\times^\ad_t\Der(\Lg))''=\langle X_2,X_3,X_4,X_5-X_9, X_6,X_7,X_8\rangle.$$
We see that the derived algebras stabilize and this Lie algebra is not solvable. On the other hand, we can compute the derived Lie algebras for the derivation Lie algebra $\Der(\Lr_{3,1})$. It is generated by $X_4,\ldots,X_9$ with the same commutation relations as above. We obtain
$$(\Der(\Lr_{3,1}))'= \langle X_4,X_5-X_9,X_6,X_7,X_8\rangle$$
and
$$(\Der(\Lr_{3,1}))''= \langle X_4,X_5-X_9,X_6,X_7,X_8\rangle.$$
So this derived series stabilizes also and $\Der(\Lr_{3,1})$ is not solvable either. But now we can conclude that the deformed Lie algebra and the direct product Lie algebra are not isomorphic. Indeed, the second derived algebra of the direct product is $5$-dimensional, because $\Lr_{3,1}$ is $2$-step solvable.
\end{ex}

\subsection{Computing the derivations of the deformed Lie algebra}

Let $\mu:\Lh\to\Lg$ be a crossed module of Lie algebras and consider again the deformed Lie algebra $\Lh\times^\mu_t\Lg$. We are interested here in its Lie algebra of derivations. But first, we record the results for the derivations of the direct product Lie algebra and the semi-direct product Lie algebra in the case where $\Lg\not=\Lh$, cf Section \ref{ss:lie_der} (where $\Lg=\Lh$).

\begin{prop}
Every derivation $D:\Lh\oplus\Lg\to\Lh\oplus\Lg$ (of the direct product Lie algebra $\Lh\oplus\Lg$) is of the
matrix form
$$D=\left(\begin{array}{cc} D_1 & D_2 \\ D_3 & D_4 \end{array}\right),$$
where $D_1$ is a derivation of $\Lh$, $D_4$ is a derivation of $\Lg$ and where $D_2\in\Hom(\Lg/[\Lg,\Lg],Z(\Lh))$, while $D_3\in\Hom(\Lh/[\Lh,\Lh],Z(\Lg))$.
\end{prop}

\begin{proof}
The proof is very similar to the proof of Proposition \ref{prop_direct_prod}.
\end{proof}

Let us now suppose that $\Lh$ and $\Lg$ are two Lie algebras with an action of $\Lg$ on $\Lh$ by derivations.

\begin{prop}
Every derivation $D:\Lh\rtimes\Lg\to\Lh\rtimes\Lg$ (of the semi-direct product Lie algebra $\Lh\rtimes\Lg$) is of the
matrix form
$$D=\left(\begin{array}{cc} D_1 & D_2 \\ D_3 & D_4 \end{array}\right),$$
where $D_4$ is a derivation of $\Lg$, $D_2$ is a derivation of $\Lg$ with values in $\Lh$ and where $D_1$, $D_2$, $D_3$ and $D_4$ satisfy the equations:
\begin{itemize}
\item[(1)] $0=D_3(h)\cdot h'-D_3(h')\cdot h$,
\item[(2)] $D_1(g'\cdot h)=D_4(g')\cdot h+g'\cdot D_1(h)$,
\item[(3)] $D_3(g'\cdot h)=[g',D_3(h)]$, i.e. $D_3$ is in the centroid (of $\Lh$ with values in $\Lg$),
\end{itemize}
\end{prop}

\begin{proof}
We compute both sides of $D[(h,g),(h',g')]=[D(h,g),(h',g')]+[(h,g),D(h',g')]$ for all $h,h'\in\Lh$ and all $g,g'\in\Lg$.
\begin{eqnarray*}
D[(h,g),(h',g')]&=&D(g\cdot h'-g'\cdot h,[g,g'])\\
&=&(D_1(g\cdot h'-g'\cdot h)+D_2[g,g'],D_3(g\cdot h'-g'\cdot h)+D_4[g,g']).
\end{eqnarray*}
On the other hand
\begin{eqnarray*}
[D(h,g),(h',g')]&=&[(D_1(h)+D_2(g),D_3(h)+D_4(g)),(h',g')] \\
&=&((D_3(h)+D_4(g))\cdot h'-g'\cdot(D_1(h)+D_2(g)),[D_3(h)+D_4(g),g']).
\end{eqnarray*}
and the last term reads
\begin{eqnarray*}
[(h,g),D(h',g')]&=&[(h,g),(D_1(h')+D_2(g'),D_3(h')+D_4(g'))] \\
&=&(g\cdot(D_1(h')+D_2(g'))-(D_3(h')+D_4(g'))\cdot h,[g,D_3(h')+D_4(g')]).
\end{eqnarray*}
Again we obtain two relations for $D$ to be a derivation:
\begin{itemize}
\item $D_1(g\cdot h'-g'\cdot h)+D_2[g,g']=(D_3(h)+D_4(g))\cdot h'-g'\cdot(D_1(h)+D_2(g))+g\cdot(D_1(h')+D_2(g'))-(D_3(h')+D_4(g'))\cdot h$, and
\item $D_3(g\cdot h'-g'\cdot h)+D_4[g,g']=[D_3(h)+D_4(g),g']+[g,D_3(h')+D_4(g')]$.
\end{itemize}
From these two relations, we obtain the following five conditions on $D_1$, $D_2$, $D_3$ and $D_4$ by putting first $g=g'=0$, then $g=h'=0$ and finally $h=h'=0$:
\begin{itemize}
\item[(1)] $0=D_3(h)\cdot h'-D_3(h')\cdot h$,
\item[(2)] $D_1(g'\cdot h)=D_4(g')\cdot h+g'\cdot D_1(h)$,
\item[(3)] $D_3(g'\cdot h)=[g',D_3(h)]$, i.e. $D_3$ is in the centroid (of $\Lh$ with values in $\Lg$),
\item[(4)] $D_2[g,g']=-g'\cdot D_2(g)+g\cdot D_2(g')$, i.e. $D_2\in\Der(\Lg,\Lh)$, and
\item[(5)] $D_4[g,g']=[D_4(g),g']+[g,D_4(g')]$, i.e. $D_4\in\Der(\Lg)$.
\end{itemize}
\end{proof}

\begin{prop}
Every derivation $D:\Lh\times^\mu_1\Lg\to\Lh\times^\mu_1\Lg$ (of the deformed Lie algebra $\Lh\times^\mu_1\Lg$ at $t=1$) is of the
matrix form
$$D=\left(\begin{array}{cc} D_1 & D_2 \\ D_3 & D_4 \end{array}\right),$$
where $D_1$, $D_2$, $D_3$ and $D_4$ satisfy the six equations:
\begin{itemize}
\item[(1)]  $D_2(\mu([h,h']))=D_3(h)\cdot h'-D_3(h')\cdot h$
\item[(2)] $D_4(\mu([h,h']))=\mu[D_1(h),h']+\mu[h,D_1(h')]$
\item[(3)] $D_1(g'\cdot h)=g'\cdot D_1(h)+D_4(g')\cdot h$
\item[(4)] $D_3(g'\cdot h)=[g',D_3(h)]+\mu[D_2(g'),h]$
\item[(5)] $D_2[g,g']=-g'\cdot D_2(g)+g\cdot D_2(g')$, i.e. $D_2\in\Der(\Lg,\Lh)$
\item[(6)] $D_4[g,g']=[D_4(g),g']+[g,D_4(g')]$, i.e. $D_4\in\Der(\Lg)$.
\end{itemize}

\end{prop}

\begin{proof}
We compute both sides of $D[(h,g),(h',g')]=[D(h,g),(h',g')]+[(h,g),D(h',g')]$ for all $h,h'\in\Lh$ and all $g,g'\in\Lg$.
\begin{eqnarray*}
D[(h,g),(h',g')]&=&D(g\cdot h'-g'\cdot h,[g,g']+\mu([h,h']))\\
&=&(D_1(g\cdot h'-g'\cdot h)+D_2([g,g']+\mu([h,h'])),D_3(g\cdot h'-g'\cdot h)+D_4([g,g']+\mu([h,h']))).
\end{eqnarray*}
On the other hand
\begin{eqnarray*}
[D(h,g),(h',g')]&=&[(D_1(h)+D_2(g),D_3(h)+D_4(g)),(h',g')] \\
&=&((D_3(h)+D_4(g))\cdot h'-g'\cdot(D_1(h)+D_2(g)),\\
&&[D_3(h)+D_4(g),g']+\mu[D_1(h)+D_2(g),h']).
\end{eqnarray*}
and the last term reads
\begin{eqnarray*}
[(h,g),D(h',g')]&=&[(h,g),(D_1(h')+D_2(g'),D_3(h')+D_4(g'))] \\
&=&(g\cdot (D_1(h')+D_2(g'))-(D_3(h')+D_4(g'))\cdot h,\\
&&[g,D_3(h')+D_4(g')]+\mu[h,D_1(h')+D_2(g')]).
\end{eqnarray*}
Again we obtain two relations for $D$ to be a derivation:
\begin{itemize}
\item $D_1(g\cdot h'-g'\cdot h)+D_2([g,g']+\mu([h,h']))=(D_3(h)+D_4(g))\cdot h'-g'\cdot(D_1(h)+D_2(g))+g\cdot (D_1(h')+D_2(g'))-(D_3(h')+D_4(g'))\cdot h$, and
\item $D_3(g\cdot h'-g'\cdot h)+D_4([g,g']+\mu([h,h']))=[D_3(h)+D_4(g),g']+\mu[D_1(h)+D_2(g),h']+[g,D_3(h')+D_4(g')]+\mu[h,D_1(h')+D_2(g')]$.
\end{itemize}
From these two relations, we obtain the following six conditions on $D_1$, $D_2$, $D_3$ and $D_4$ by putting first $g=g'=0$, then $g=h'=0$ and finally $h=h'=0$:
\begin{itemize}
\item[(1)]  $D_2(\mu([h,h']))=D_3(h)\cdot h'-D_3(h')\cdot h$
\item[(2)] $D_4(\mu([h,h']))=\mu[D_1(h),h']+\mu[h,D_1(h')]$
\item[(3)] $D_1(g'\cdot h)=g'\cdot D_1(h)+D_4(g')\cdot h$
\item[(4)] $D_3(g'\cdot h)=[g',D_3(h)]+\mu[D_2(g'),h]$
\item[(5)] $D_2[g,g']=-g'\cdot D_2(g)+g\cdot D_2(g')$, i.e. $D_2\in\Der(\Lg,\Lh)$
\item[(6)] $D_4[g,g']=[D_4(g),g']+[g,D_4(g')]$, i.e. $D_4\in\Der(\Lg)$.
\end{itemize}
\end{proof}

\begin{ex}
For the crossed module $0:V\to\Lg$ associated to a $\Lg$-module $V$, we have that the derivations of
$V\times^0_1\Lg$ and of $V\rtimes\Lg$ coincide, but are in general different from those of $V\oplus\Lg$.
\end{ex}

\begin{ex}
Let $\Lg= \langle x,y\,|\,[x,y]=x\rangle$ be the two-dimensional non-abelian complex Lie algebra. Consider the crossed module $i:\C x\to\Lg$ given by the inclusion of the ideal generated by $x$ into $\Lg$, and the corresponding Lie algebra $\C x\rtimes^i_1\Lg$.  One computes that the above conditions (1), (2) and (5) are trivially true for
\begin{itemize}
\item $D_1:\C x\to\C x$, $D_1(x)=\lambda x$,
\item $D_2:\Lg\to\C x$, $D_2(x)=\alpha x$, $D_2(y)=\beta x$,
\item $D_3:\C x\to\Lg$, $D_3(x)=\mu x+\nu y$,
\item $D_4:\Lg\to\Lg$, $D_4(x)=ax+cy$, $D_4(y)=bx+dy$,
\end{itemize}
for some scalars $\alpha,\beta,\lambda,\mu,\nu,a,b,c,d$. Condition (3) implies that $D_4(y)$ has no component
in $y$, i.e. $d=0$. Condition (4) implies that $\nu=0$. Condition (6) (i.e. the fact that $D_4$ is a derivation) implies that $c=0$. Thus the Lie algebra of derivations of $\C x\rtimes^i_1\Lg$ is of dimension $6$ (described by the six scalars $\alpha,\beta,\lambda,\mu,a,b$). The Lie algebra of derivations of the semi-direct product $\C x\rtimes\Lg$ is also of dimension $6$. But the Lie algebra of derivations of the direct product is of dimension $4$.
\end{ex}

\begin{ex}
Let $\Lg= \langle e_1, e_2, e_3, e_4\,|\, [e_1, e_3]= e_3, [e_2, e_4]= e_4 \rangle$ and $\Lh= \langle e_1, e_3, e_4\,|\, [e_1, e_3]= e_3\rangle$ is a Lie subalgebra of 
$\Lg$. Consider the crossed module $i:\Lh \to \Lg$ given by the inclusion. Then $\dim (\Der(\Lh \rtimes^i_1 \Lg))= 10 < 11= \dim (\Der(\Lh \rtimes \Lg))$.
\end{ex}


\begin{thebibliography}{50}

\bibitem{RAS} Alemi, Mohammad Reza; Saeedi, Farshid
\emph{Derivation algebra of semi-direct sum of Lie algebras.}
Asian-Eur. J. Math. {\bf 15} (2022), no.4, Paper No. 2250063, 9 pp.

\bibitem{BW} Burde, Dietrich; Wagemann, Friedrich
\emph{Sympathetic Lie algebras and adjoint cohomology for Lie algebras.}
J. Algebra {\bf 629} (2023), 381--398.



\bibitem{GK}
Goze, Michel; Khakimdjanov, Yusupdjan \emph{Nilpotent Lie algebras}. Mathematics and Its Applications {\bf 361}. Springer Dordrecht (1996). 

\bibitem{LL} Leger, George F.; Luks, Eugene M.
\emph{Generalized derivations of Lie algebras.}
J. Algebra {\bf 228} (2000), no.1, 165--203.


\bibitem{Ra} Rauch, G\'erard
\emph{Effacement et d\'eformation}
Ann. Inst. Fourier (Grenoble) {\bf 22} (1972), no.1, 239--269.


\bibitem{Ri} Richardson, R. W., Jr.
\emph{On the rigidity of semi-direct products of Lie algebras.}
Pacific J. Math. {\bf 22} (1967), 339--344.




\end{thebibliography}
\end{document}